\documentclass[preprint,11pt,3p]{elsarticle}
\usepackage{amsmath,amsthm,amsfonts,amssymb,arydshln}
\usepackage{xcolor}
\usepackage{algorithm}
\usepackage{multirow}
\usepackage{multicol}
\usepackage{booktabs}
\usepackage{algorithmic}
\usepackage{comment}
\usepackage{threeparttable}
\usepackage{tikz}
\usetikzlibrary{arrows.meta, positioning}

\usepackage{mathrsfs}
\usepackage{sectsty}
\definecolor{astral}{RGB}{46,116,181}
\sectionfont{\color{astral}}
\newtheorem{theorem}{Theorem}[section]
\newtheorem{lemma}[theorem]{Lemma}
\newtheorem{corollary}[theorem]{Corollary}

\newtheorem{definition}[theorem]{Definition}

\newtheorem{remark}[theorem]{Remark}
\usepackage{hyperref}
\usepackage{subfigure}
\usepackage{scalerel}
\usepackage{stackengine,wasysym}

\usepackage{tikz}
\pgfdeclarelayer{bg} 
\pgfsetlayers{bg,main}

\usepackage{tikz,xcolor}
\definecolor{lime}{HTML}{A6CE39}
\definecolor{lightblue}{rgb}{0.0, 0.0, 0.5}
\DeclareRobustCommand{\orcidicon}{%
	\begin{tikzpicture}
	\draw[lime, fill=lime] (0,0)
	circle [radius=0.16]
	node[white] {{\fontfamily{qag}\selectfont \tiny ID}};
	\draw[white, fill=white] (-0.0625,0.095)
	circle [radius=0.007];
	\end{tikzpicture}
	\hspace{-2mm}
}
\foreach \x in {A, ..., Z}{%
	\expandafter\xdef\csname orcid\x\endcsname{\noexpand\href{https://orcid.org/\csname orcidauthor\x\endcsname}{\noexpand\orcidicon}}
}

\newcommand{\D}{{\mathrm D}}
\newcommand{\F}{{\mathrm F}}

\newcommand{\Ind}{{\mathrm {Ind}}}

\begin{document}
\begin{frontmatter}
\title{ 
On the Dual Drazin Inverse of Adjacency Matrices of Dual-number-Weighted Digraphs
}
\author[a]{Yue Zhao}
\ead{yuezhao0303@163.com}

\author[b]{Daochang Zhang\corref{cor1}}
\ead{daochangzhang@126.com}
\cortext[cor1]{Corresponding author.}

\author[a]{Zhongshan Li}
\ead{zli@gsu.edu}

\author[a]{Frank J. Hall}
\ead{fhall@gsu.edu}

\address[a]{Department of Math \& Stat, Georgia State University, Atlanta, GA 30303, USA}
\address[b]{College of Sciences, Northeast Electric Power University, Jilin, P.R. China}

\begin{abstract}
The motivation of this paper is to investigate the dual Drazin inverse of adjacency matrices arising from several classes of connected dual-number-weighted digraphs over the dual complex algebra. Explicit formulas for the dual Drazin inverse of dual complex anti-triangular block matrices are derived under suitable assumptions. These results are applied to DN-DS digraphs, DN-DLS digraphs, and DN-DW digraphs. In particular, the assumptions in \cite{AMPMJM2026} are weakened for DN-DS digraphs, an open problem in \cite{AMPMJM2026} for the case $BC=0$ is generalized and solved for DN-DLS digraphs. And the group inverse result in \cite{MNSEJAA2022} for bipartite block form adjacency matrices is extended to the dual Drazin inverse for DN-DW digraphs. We further derive explicit formulas for the dual group inverse and dual Drazin inverse of another adjacency matrix for DN-DW digraphs.
\end{abstract}


\end{frontmatter}

\section{Introduction}
Dual numbers, introduced by Clifford in 1871 \cite{Clifford1871}, provide an algebraic framework for representing quantities consisting of a standard part and an infinitesimal part. The infinitesimal component can be interpreted as a perturbation of the standard part or as derivative information in time-dependent models. Dual numbers and their algebra have been widely used in kinematic analysis \cite{Kinematic}, screw motion theory \cite{Screw}, and rigid body motion analysis \cite{Rigidbody2,Rigidbody}, spatial mechanisms \cite{Spatial}. Consequently, increasing attention has been devoted to the fundamental theoretical properties of dual matrices.

Dual matrices, whose entries are dual numbers,
provide a convenient mathematical model for describing coupled positional and velocity information in mechanical systems. It also has application in machine vision, robotics, and traveling wave identification and brain dynamics \cite{Robotics,Rigidbody2,Kinematic,Brain}.

Generalized inverses of dual matrices arise naturally in the analysis and synthesis of mechanisms and in various kinematic computation problems. In 2018, Pennestri\`{\i} et al. \cite{Kinematic} investigated the Moore-Penrose dual generalized inverse and its application to Kinematic synthesis of spatial linkages. In 2021, Udwadia \cite{Udwadia2021} applied dual generalized inverse to solve systems of linear dual equations. However, unlike the classical matrix case, dual matrices do not always admit generalized inverses due to their special algebraic structure. This motivates the study of existence conditions, representations, and computational methods for dual generalized inverses. In 2020, Udwadia et al. \cite{Udwadia2020} investigated the question of whether all dual matrices have dual Moore-Penrose generalized inverses. In 2023, Zhong and Zhang \cite{ZZFil2023} proved the uniqueness of the dual Drazin
inverse whenever it exists and established necessary and sufficient conditions for a square dual matrix to admit a dual Drazin inverse, together with a compact computational formula.

A fundamental problem in matrix theory is to characterize the invertibility of matrices associated with graphs and to derive explicit formulas for their inverses or generalized inverses in terms of graph structure. Research has also focused on generalized inverses of graph-associated matrices. In particular, the group and Drazin inverses have been studied for several structured graph families such as path graphs, broom tree, bipartite digraphs, and multi-star graph in 
\cite{AMPMJM2026,Group path2008,Catral2009,broom tree,MNSEJAA2022}. These developments highlight the role of generalized inverses in revealing structural features of graphs and digraphs, connecting combinatorial characteristics with algebraic and spectral properties.

Motivated by the two research directions highlighted above, we investigate dual generalized inverses of structured dual block matrices arising from graph models. In particular, we derive new formulas for the dual Drazin inverse of dual complex anti-triangular block matrices and apply these results to adjacency matrices of special graph classes, which are devoted to an open problem that was first proposed by Campbell and Meyer \cite{Campbell1991} for complex matrices, and has since attracted considerable attention \cite{HWWLAA2001,ZMD2026,daochangjcam2025}. Explicit representations of the generalized inverses of adjacency matrices help clarify the relationship between their associated graph topology and matrix decomposition, which are useful in the analysis of spectral properties, structural perturbations, and algebraic invariants of directed graphs. Moreover, when edge weights are described by dual numbers, the corresponding dual generalized inverse of the adjacency matrix captures both structural information and infinitesimal perturbations, which is relevant to sensitivity analysis in weighted network models.

The remainder of this paper is organized as follows. In Section 2, we present preliminaries on dual numbers, the dual Drazin inverse, and the adjacency matrix of a connected digraph. Section 3 derives explicit expressions for the dual Drazin inverse of dual complex anti-triangular block matrices under certain assumptions. In Section 4, we first provide the notion of a connected dual-number-weighted digraph. Then we introduce DN-DS digraphs, DN-DLS digraphs, DN-DW digraphs, and investigate their adjacency matrices. Moreover, we apply the main results in Section 3 to adjacency matrices of the above graph classes. For DN-DS digraphs, we derive explicit representations of the corresponding dual Drazin inverses by removing half of the assumptions imposed in \cite{AMPMJM2026} and extending the results from $\mathbb{C}$ to the dual complex algebra $\mathbb{DC}_z$. Furthermore, for DN-DLS digraphs, we generalize and resolve an open problem in \cite{AMPMJM2026}, that is, to derive the representation for the representation of the Drazin inverse in the case $BC=0$. In addition, we derive some explicit representations of the dual group inverse and dual Drazin inverse for two kinds of adjacency matrices for DN-DW digraphs, which generalize the group inverse of the bipartite block form adjacency matrix of the Dutch windmill graph mentioned in \cite{MNSEJAA2022}.

\section{Preliminaries and key lemmas}
Let $\mathbb{D}$ denote the set of dual numbers. 
A dual number $\hat{p} \in \mathbb{D}$ is defined by
$\hat{p} = p + \varepsilon p_0$,
where $p,p_0\in \mathbb{R}$ and $\varepsilon$ is the hypercomplex unit basis satisfying
\[
\varepsilon \neq 0, \quad \varepsilon^2 = 0, \quad
0\varepsilon = \varepsilon 0 = 0, \quad
1\varepsilon = \varepsilon 1 = \varepsilon.
\]
The scalars $p$ and $p_0$ are called the \emph{standard part} and 
\emph{infinitesimal part} of $\hat{p}$, respectively. 
A dual number is said to be \emph{appreciable} if $p \neq 0$, 
and \emph{infinitesimal} otherwise.

When $p, p_0$ are replaced by complex numbers, we obtain the 
\emph{dual complex field}
$\mathbb{DC} = \{ z + \varepsilon z_0 \mid z, z_0 \in \mathbb{C} \}.$
For vectors $x, x_0 \in \mathbb{C}^n$, a \emph{dual complex vector} is defined as
$\widehat{x} = x + \varepsilon x_0$, and the set of all such vectors is denoted by $\mathbb{DC}^n$. Similarly, for matrices $A, A_0 \in \mathbb{C}^{m\times n}$, a 
\emph{dual complex matrix} is defined as
$\widehat{A} = A + \varepsilon A_0$, and the set of all $m \times n$ dual complex matrices is denoted by
$\mathbb{DC}^{m\times n}$. 
The matrices $A$ and $A_0$ are called the standard part and infinitesimal part of $A$, respectively.

For a matrix $A\in \mathbb{C}^{n\times n}$, the minimum nonnegative integer $k$ satisfying rank($A^k$)=rank($A^{k+1}$) is called the index of $A$, denoted by $\Ind(A)$. In \cite{Wang2024,ZZFil2023}, the dual Drazin inverse of $\widehat{A}$ was defined and it was shown that it exists if and only if
$(I_n - AA^D) M (I_n - AA^D) = O$, where $M=\sum_{i=1}^{k}A^{k-i}A_0A^{i-1}$ such that $\Ind(A)=k$ and $A^D$ is the usual Drazin inverse of the complex matrix $A$. Hence, we denote
\[
\mathbb{DC}_z^{n\times n}=\{\widehat{A}|\widehat{A}=A+\varepsilon A_0\in \mathbb{DC}^{n\times n}, (I_n-AA^D)M(I_n-AA^D)=O\},
\] that is to say, the set of all matrices $\widehat{A}$ for which the dual Drazin inverse exists.

Let $\widehat{A} = A + \varepsilon A_0 \in \mathbb{DC}_z^{n \times n}$ with $\Ind(A)=k$. The dual Drazin inverse of $\widehat{A}$ 
is the unique matrix $\widehat{X} \in \mathbb{DC}_z^{n \times n}$ 
satisfying
\begin{equation}\label{Def Dual Drazin inverse}
\widehat{A}^{k}\widehat{X}\widehat{A}
= \widehat{A}^k, \quad
\widehat{X}\widehat{A}\widehat{X}
= \widehat{X}, \quad
\widehat{A}\widehat{X}
= \widehat{X}\widehat{A}.
\end{equation}
It is denoted by $\widehat{A}^D$. When $k=1$, $\widehat{A}^D$ reduces to the dual group inverse of $\widehat{A}$, denoted by $\widehat{A}^{\#}$. The spectral idempotent of $\widehat{A}$ corresponding to $\{0\}$ is 
defined by $\widehat{A}^\pi = I - \widehat{A}^e$,
where $\widehat{A}^e = \widehat{A}\widehat{A}^D$. Here, $I$ and $O$ denote the identity and zero matrices in $\mathbb{DC}_z^{n\times n}$, respectively. Since $(\widehat{A}^D)^s = (\widehat{A}^s)^D$, we write $(\widehat{A}^D)^s = \widehat{A}^{Ds} = \widehat{A}^{sD}$ for any positive integer $s$.
Throughout the paper, a sum is defined to be zero whenever the lower limit exceeds the upper limit, and we adopt the convention $\widehat{A}^0 = I$.

The formulation in \eqref{Def Dual Drazin inverse} appears in the literature as such, see \cite{Wang2024}. Now, using the first and third 
equations in (1), we can rewrite the first equation as   $\widehat{A}^{k + 1} \widehat{X}  =  \widehat{A}^k$
so that the dual Drazin inverse of $\widehat{A}$ is formally the same as the usual Drazin inverse of $A$. Since $k=\operatorname{Ind}(A)$, it is natural to consider some appropriate definition of the index for $\widehat{A}$ and to investigate its relation to $\operatorname{Ind}(A)$, respectively. 

\begin{remark}
Let $\widehat{A}=A+\varepsilon A_0\in \mathbb{DC}^{n\times n}$. 
Define the standard rank by 
\[
\operatorname{rank}_s(\widehat{A})
:= \operatorname{rank}(A).
\]
Since 
\[
\widehat{A}^k
= A^k
+ \varepsilon \sum_{i=1}^{k} A^{k-i} A_0 A^{i-1},
\]
the standard part of $\widehat{A}^k$ is $A^k$, and thus
\[
\operatorname{rank}_s(\widehat{A}^k)
= \operatorname{rank}(A^k).
\]
Defining
\[
\operatorname{Ind}_s(\widehat{A})=k
\;\Longleftrightarrow\;
\operatorname{rank}_s(\widehat{A}^k)
=\operatorname{rank}_s(\widehat{A}^{k+1}),
\]
we obtain the appreciable index
\[
\operatorname{Ind}_s(\widehat{A})
=\operatorname{Ind}(A),
\]which is also mentioned in \cite[Definition 3.1]{Wang2024} over $\mathbb{DR}^{n\times n}$.

On the other hand, let $\widehat{A}=A+\varepsilon A_0 \in \mathbb{DC}^{n\times n}$ with 
$\operatorname{Ind_s}(\widehat{A})=k$. 
Then the smallest positive integer $t~(k\le t\le 2k)$ for which
\[
\operatorname{rank}_s\!\left(\widehat{A}^{\,t}\right)
=
\operatorname{rank}\!\left(\widehat{A}^{\,t}\right)
\]
holds is called the \emph{dual index} of $\widehat{A}$, and is denoted by 
$\operatorname{Ind}(\widehat{A})=t$, which is also mentioned in \cite[Definition 3.2]{Wang2024} over $\mathbb{DR}^{n\times n}$.
\end{remark}

Moreover, the isomorphism index of a dual matrix can be defined as follows.
\begin{remark}
Let $\widehat{A}=A+\varepsilon B\in \mathbb{DC}^{n\times n}$ with $\operatorname{Ind_s}(\widehat{A})=k$, 
where $A,B\in\mathbb{C}^{n\times n}$. 
Consider the block matrix
\[
\Phi(\widehat{A})
=
\begin{pmatrix}
A & B \\
0 & A
\end{pmatrix},
\]
and define the isomorphism index of $\widehat{A}$ as follows
\[
\operatorname{Ind_{\Phi}}(\widehat{A})
:= \operatorname{Ind}(\Phi(\widehat{A})).
\]
By \cite[Theorem 7.7.2]{Campbell1991}, we conclude that 
$
\operatorname{Ind}(A)\le \operatorname{Ind_{\Phi}}(\widehat{A})
\le 2\operatorname{Ind}(A)
$, i.e.,\[k\le \operatorname{Ind_{\Phi}}(\widehat{A})\le 2k.\]
\end{remark}

We now introduce the notation in \cite{ZZFil2023} that will be used in the sequel.

Assume that $\widehat{A}=A+\varepsilon A_0\in \mathbb{DC}_z^{n \times n}$, $\operatorname{Ind}(A)=k$. As already observed,
\[
\widehat{A}^{k}=A^{k}+\varepsilon M,~
\text{where}~
M=\sum_{i=1}^{k} A^{k-i}A_0A^{i-1}.
\] From \cite{ZZFil2023}, we obtain
\[\widehat{A}^{D}=A^{D}+\varepsilon A_R,
\]
where
$A_R=-A^{D}A_0A^{D}
+\sum\limits_{i=0}^{k-1}(A^{D})^{i+2}A_0A^{i}(I_n-AA^{D})
+\sum\limits_{i=0}^{k-1}(I_n-AA^{D})A^{i}A_0(A^{D})^{i+2}.
$
Meanwhile, by induction, we can  obtain $(\widehat{A}^D)^k=A^{kD}
+\varepsilon\sum\limits_{i=0}^{k-1}A^{kD}A_{R}A^{(k-i-1)D}$, that will be used in Section 4.

We next present Cline's formula for dual Drazin inverses, extending the classical result in \cite{cline1965}, which will be used in Section 3. 

\begin{lemma}\label{cline}
For $\widehat{A}\in\mathbb{DC}^{m\times n}$, $\widehat{B}\in\mathbb{DC}^{n\times m}$ and $\widehat{B}\widehat{A}\in\mathbb{DC}_z^{n\times n}$,
$(\widehat{A}\widehat{B})^D = \widehat{A}(\widehat{B}\widehat{A})^{2D}\widehat{B}.$
\end{lemma}
\begin{proof}
    Suppose that $\operatorname{Ind}_s(\widehat{A}\widehat{B})=k=\Ind_s(\widehat{B}\widehat{A})+1$, then
    \begin{align*}
        &(\widehat A\widehat  B)^D\widehat A\widehat B(\widehat A\widehat B)^D=\widehat{A}(\widehat{B}\widehat{A})^{2D}\widehat{B}\widehat A\widehat  B\widehat{A}(\widehat{B}\widehat{A})^{2D}\widehat{B}=\widehat{A}(\widehat{B}\widehat{A})^{2D}\widehat{B}=(\widehat A\widehat  B)^D,\\
        &(\widehat A\widehat  B)(\widehat A\widehat  B)^D=\widehat A\widehat  B\widehat{A}(\widehat{B}\widehat{A})^{2D}\widehat{B}=\widehat A(\widehat{B}\widehat{A})^{2D}(\widehat{B}\widehat{A})\widehat{B}=(\widehat A\widehat  B)^D(\widehat A\widehat  B),\\
        &(\widehat A\widehat  B)^{k+1}(\widehat A\widehat  B)^D =\widehat A(\widehat B \widehat A)^k \widehat B\widehat{A}(\widehat{B}\widehat{A})^{2D}\widehat{B}
        =\widehat A(\widehat{B}\widehat{A})^k
        (\widehat{B}\widehat{A})^D\widehat  B
        =\widehat A(\widehat B \widehat A)^{k-1}\widehat  B
        =(\widehat A \widehat B)^k.
        \end{align*}
        Therefore, by \eqref{Def Dual Drazin inverse}, we finish the proof.
\end{proof}

\begin{lemma}\label{dualtri}
    Let $\widehat W=\begin{pmatrix}
        \widehat A&\widehat B\\O&\widehat D
    \end{pmatrix}$ and $\widehat V=\begin{pmatrix}
        \widehat D&O\\\widehat B&\widehat A
    \end{pmatrix}\in\mathbb {DC}_z^{n\times n}$, where $\widehat A$ and $\widehat D$ are dual Drazin invertible such that $\Ind_s(\widehat A)=p$ and $\Ind_s(\widehat D)=q$. Then
    \begin{align*}
        \widehat W^D=\begin{pmatrix}
            \widehat A^D&\widehat S\\O&\widehat D^D
        \end{pmatrix}\quad\text{and}\quad \widehat V^D=\begin{pmatrix}
            \widehat D^D&O\\\widehat S&\widehat A^D
        \end{pmatrix},
    \end{align*}where
    \begin{align*}
        \widehat S=\sum_{i=0}^{q-1}\widehat A^{(i+2)D}\widehat B\widehat D^i\widehat D^\pi+\widehat A^\pi\sum_{i=0}^{p-1}\widehat A^i\widehat B\widehat D^{(i+2)D}-\widehat A^D\widehat B\widehat D^D.
    \end{align*}
\end{lemma}
\begin{proof}
    We first verify that $\widehat W^D$ satisfies the three defining equations \eqref{Def Dual Drazin inverse} of the dual Drazin inverse.
\medskip
\noindent
(i) A routine block computation shows that
    \begin{align*}
        \widehat W^D\widehat W\widehat W^D
        &=\begin{pmatrix}
            \widehat A^D&\sum\limits_{i=0}^{q-1}\widehat A^{(i+2)D}\widehat B\widehat D^i\widehat D^\pi+\widehat A^\pi\sum\limits_{i=0}^{p-1}\widehat A^i\widehat B\widehat D^{(i+2)D}-\widehat A^D\widehat B\widehat D^D\\O&\widehat D^D
        \end{pmatrix}=\widehat W^D,
    \end{align*}
\medskip
\noindent
(ii) By direct multiplication we obtain
    \begin{align*}
        \widehat W \widehat W^D
        &=\begin{pmatrix}
            \widehat A\widehat A^D&\sum\limits_{i=0}^{q-1}\widehat A^{(i+1)D}\widehat B\widehat D^i\widehat D^\pi+\widehat A^\pi\sum\limits_{i=0}^{p-1}\widehat A^{i}\widehat B\widehat D^{(i+1)D}\\O&\widehat D\widehat D^D
        \end{pmatrix},
    \end{align*}
    and
    \begin{align*}
        \widehat W^D \widehat W
        &=\begin{pmatrix}
            \widehat A\widehat A^D&\widehat A^D\widehat B+\sum\limits_{i=0}^{q-1}\widehat A^{(i+2)D}\widehat B\widehat D^{i+1}\widehat D^\pi+\widehat A^\pi\sum\limits_{i=0}^{p-1}\widehat A^{i}\widehat B\widehat D^{(i+1)D}-\widehat A^D\widehat B\widehat D^D\widehat D\\O&\widehat D\widehat D^D
        \end{pmatrix}.
    \end{align*}
    After adjusting appropriately the upper and lower limits of the summations, we conclude that $ \widehat W^D \widehat W= \widehat W \widehat W^D$.
    
\medskip
\noindent
(iii) It is well known that for any positive integer $n$
\[
\widehat W^n=
\begin{pmatrix}
\widehat A^n & \widehat Z_n\\
O & \widehat D^n
\end{pmatrix},
\qquad 
\widehat Z_n=\sum_{i=0}^{n-1}\widehat A^{n-1-i}\widehat B\widehat D^i.
\]
Let $k=\Ind_s(\widehat W)\ge p+q$. Then
\begin{equation}\label{index-rel}
\widehat A^{k+1}\widehat A^D=\widehat A^k, 
\qquad 
\widehat D^{k+1}\widehat D^D=\widehat D^k .
\end{equation}
Furthermore, we compute
\[
\widehat W^{k+1}\widehat W^D
=
\begin{pmatrix}
\widehat A^{k+1} & \widehat Z_{k+1}\\
O & \widehat D^{k+1}
\end{pmatrix}
\begin{pmatrix}
\widehat A^D & \widehat S\\
O & \widehat D^D
\end{pmatrix}=
\begin{pmatrix}
\widehat A^{k+1}\widehat A^D &
\widehat A^{k+1}\widehat S+\widehat Z_{k+1}\widehat D^D\\
O &
\widehat D^{k+1}\widehat D^D
\end{pmatrix}.
\]
Utilizing \eqref{index-rel}, we obtain the $(1,1)$-entry and $(4,4)$-entry of $W^{k+1}W^D$ are $A^k$ and $D^k$, respectively.
Hence it remains to show that
\[
\widehat A^{k+1}\widehat S+\widehat Z_{k+1}\widehat D^D=\widehat Z_k.
\]
Since
$\widehat Z_{k+1}=\sum\limits_{i=0}^{k}\widehat A^{k-i}\widehat B\widehat D^i,$
it follows that
$\widehat Z_{k+1}\widehat D^D=\sum\limits_{i=0}^{k}\widehat A^{k-i}\widehat B\widehat D^i\widehat D^D.$
We split the summation into two parts:
\[
\widehat Z_{k+1}\widehat D^D
=\sum_{i=0}^{q}\widehat A^{k-i}\widehat B\widehat D^i\widehat D^D
+\sum_{i=q+1}^{k}\widehat A^{k-i}\widehat B\widehat D^i\widehat D^D.
\]
For $i\ge q+1$, since $\operatorname{Ind}_s(\widehat D)=q$, we have
$D^iD^D=D^{i-1}.$
Thus
\[
\sum_{i=q+1}^{k}\widehat A^{k-i}\widehat B\widehat D^{i-1}=
\sum_{i=q}^{k-1}\widehat A^{k-1-i}\widehat B\widehat D^i.
\]
Consequently,
\begin{align}\label{ZD}
\widehat Z_{k+1}\widehat D^D
=\sum_{i=0}^{q}\widehat A^{k-i}\widehat B\widehat D^i\widehat D^D
+\sum_{i=q}^{k-1}\widehat A^{k-1-i}\widehat B\widehat D^i.
\end{align}
From $\Ind_s(\widehat A)=p$, it follows that
$
\widehat A^{p+1}\widehat A^D=\widehat A^p
~\text{and}~
\widehat A^p\widehat A^\pi=O .
$
Consequently, for any integer $k\ge p$ we have
\begin{equation*}
\widehat A^{k+1}\widehat A^D=\widehat A^k,
\qquad
\widehat A^{k+1}\widehat A^\pi=O .
\end{equation*}
It follows that
\begin{equation*}
\widehat A^{k+1}(\widehat A^D)^t=\widehat A^{k+1-t},
\qquad
1\le t\le k+1-p .
\end{equation*}
Using the explicit form of $\widehat S$ and the above identities, we simplify the expression $\widehat A^{k+1}\widehat S$ as follows:

\begin{align}\label{AS}
\widehat A^{k+1}\widehat S&=\widehat A^{k+1}\sum_{i=0}^{q-1}\widehat A^{(i+2)D}\widehat B\widehat D^i\widehat D^\pi
+\widehat A^{k+1}\widehat A^\pi\sum_{i=0}^{p-1}\widehat A^i\widehat B\widehat D^{(i+2)D}
-\widehat A^{k+1}\widehat A^D\widehat B\widehat D^D \notag\\
&=\sum_{i=0}^{q-1}\widehat A^{k+1}(\widehat A^D)^{i+2}\widehat B\widehat D^i\widehat D^\pi
+\underbrace{\widehat A^{k+1}\widehat A^\pi}_{=\,O}\sum_{i=0}^{p-1}\widehat A^i\widehat B\widehat D^{(i+2)D}
-\underbrace{\widehat A^{k+1}\widehat A^D}_{=\,\widehat A^k}\widehat B\widehat D^D \notag\\
&=\sum_{i=0}^{q-1}\widehat A^{k-1-i}\widehat B\widehat D^i\widehat D^\pi
-\widehat A^k\widehat B\widehat D^D\notag\\
&=\sum_{i=0}^{q-1}\widehat A^{k-1-i}\widehat B\widehat D^i
-\sum_{i=0}^{q}\widehat A^{k-i}\widehat B\widehat D^i\widehat D^D.
\end{align}
Adding the two expressions \eqref{AS} and \eqref{ZD} gives
\[
\widehat A^{k+1}\widehat S+\widehat Z_{k+1}\widehat D^D
=\sum_{i=0}^{k-1}\widehat A^{k-1-i}\widehat B\widehat D^i
=\widehat Z_k.
\]
Therefore $\widehat W^D$ satisfies
\eqref{Def Dual Drazin inverse}.  Similarly, $\widehat V^D$ also satisfies \eqref{Def Dual Drazin inverse}, and the proof is complete.
\end{proof}

The Drazin inverse of the sum of two matrices satisfying $PQ=O$ over 
$\mathbb{C}$ was characterized in \cite{HWWLAA2001}. We next extend 
this result to $\mathbb{DC}_z$, which will also be used in Section 3. The proof is analogous to the complex case, and hence we only sketch the main steps.

\begin{lemma}\label{PQ=0}
Let $\widehat{P}, \widehat{Q}\in\mathbb{DC}_z^{n\times n}$ be dual complex matrices such that $\Ind_s(\widehat{P})=r$, $\Ind_s(\widehat{Q})=t$. If $\widehat{P}\widehat{Q}=O$, then
\begin{align*}
    (\widehat{P}+\widehat{Q})^{D} 
&=\widehat{Q}^{\pi}\sum_{i=0}^{t-1}\widehat{Q}^i(\widehat{P}^D)^{i+1} 
+ \sum_{i=0}^{r-1} (\widehat{Q}^{D})^{i+1}\widehat{P}^i\widehat{P}^{\pi}. 
\end{align*}
\end{lemma}
\begin{proof}
    By Lemma \ref{cline}, we obtain
    \begin{align}\label{clineproof}
        (\widehat {P} + \widehat{Q})^D =
\left(\begin{bmatrix}
     I&\widehat{Q}
\end{bmatrix}
\begin{bmatrix}
\widehat{P} \\
\widehat{I}
\end{bmatrix}
\right)^{D} =
\begin{bmatrix}
    I&\widehat{Q}
\end{bmatrix}
\left(
\begin{bmatrix}
\widehat{P} &{O}\\
{I} & \widehat{Q}
\end{bmatrix}^{D}
\right)^2
\begin{bmatrix}
\widehat{P} \\
I
\end{bmatrix}.
\end{align}
    Moreover, utilizing Lemma \ref{dualtri}, we derive
    \begin{align}\label{clinepr}
       \begin{bmatrix}
\widehat{P} &{O}\\
{I} & \widehat{Q}
\end{bmatrix}^{D}=\begin{bmatrix}
\widehat {P}^D &{O}\\
\sum\limits_{i=0}^{r-1}\widehat Q^{(i+2)D}\widehat P^i\widehat P^\pi+\sum\limits_{i=0}^{t-1}\widehat Q^\pi \widehat Q^i\widehat P^{(i+2)D}-\widehat Q^D\widehat P^D & \widehat{Q}^D
\end{bmatrix}.
    \end{align}
    Substituting \eqref{clinepr} into \eqref{clineproof} yields the desired representation of $(\widehat{P}+\widehat{Q})^{D}$, and the proof is complete.
\end{proof}

The preceding basic content on dual numbers and the dual Drazin inverse provides the algebraic framework for studying graph-associated matrices 
over $\mathbb{DC}_z$. We next recall basic graph-theoretic definitions needed for our analysis.

Let $G$ be a simple undirected graph with vertex set 
$V=\{v_1,\ldots,v_n\}$. The adjacency matrix of $G$ is the 
$n\times n$ matrix $A=(a_{ij})$ defined by
\[
a_{ij}=
\begin{cases}
1, & \text{if } v_iv_j\in E(G),\\
0, & \text{otherwise}.
\end{cases}
\]
The graph $G$ is called nonsingular (singular) if $A$ is 
nonsingular (singular). The adjacency matrix of a simple 
undirected graph is symmetric.

Let $A=(a_{ij})\in\mathbb{C}^{n\times n}$. The digraph 
associated with $A$, denoted by $D(A)=(V,E)$, is defined by
\[
(v_i,v_j)\in E \quad \text{if and only if } a_{ij}\neq 0.
\]
The digraph $D(A)$ is called a tree digraph if it is strongly 
connected and all its cycles have length $2$. In this case, 
$A$ is called combinatorially symmetric.

\section{Main results}

In this section, we establish explicit representations for the dual Drazin inverse of dual complex anti-triangular block matrices under appropriate assumptions, thereby extending and generalizing several existing results in the literature.
\begin{theorem}\label{DecompABIO1}
Let $\widehat{N}=\begin{pmatrix}
    \widehat{A}&\widehat B\\I&O
\end{pmatrix}\in\mathbb{DC}_z^{n\times n}$ be an anti-triangular block matrix, where $\widehat A$ and $\widehat B$ are of the same order and dual Drazin invertible. Assume that $\Ind_s(\widehat A)=i_{\widehat A}$, $\Ind_s(\widehat B)=i_{\widehat B}$ and the conditions $\widehat A \widehat A^\pi \widehat B=\widehat B\widehat A\widehat A^\pi$ and $\widehat A\widehat A^e\widehat B=O$ hold. Then
\begin{align*}
 \widehat{N}^D=\begin{pmatrix}
     \sum\limits_{i=0}^{i_{\widehat B}-1}\widehat B^\pi \widehat B^i\widehat A^{(2i+1)D}&\widehat B^e\\
     \sum\limits_{i=0}^{i_{\widehat B}-1}\widehat B^\pi \widehat B^i\widehat A^{(2i+2)D}+\widehat B^D\widehat A^\pi&-\widehat A\widehat A^\pi \widehat B^D
 \end{pmatrix}.
\end{align*}
\end{theorem}

\begin{proof}
Consider the following decomposition of the matrix $\widehat {N}$ as follows:
\begin{align*}
   \widehat M=\begin{pmatrix}
        \widehat A\widehat A^e&O\\O&O
    \end{pmatrix}+\begin{pmatrix}
        \widehat A\widehat A^\pi&\widehat B\\I&O
    \end{pmatrix}:=\widehat U+\widehat V.
\end{align*}
Under the condition $\widehat A\widehat A^\pi \widehat B=\widehat B\widehat A\widehat A^\pi$, the dual Drazin inverse of $\widehat V$ admits the following explicit representation.
\begin{align*}
    \widehat V^D=\begin{pmatrix}
        \widehat A\widehat A^\pi&\widehat B\\I&O
    \end{pmatrix}^D=\begin{pmatrix}
        O&\widehat B^e\\\widehat B^D&-\widehat A\widehat A^\pi \widehat B^D
    \end{pmatrix},
\end{align*} and hence
\begin{align*}
    \widehat V^\pi=I-\widehat V\widehat V^D=\begin{pmatrix}
        \widehat B^\pi&O\\O&\widehat B^\pi
    \end{pmatrix}.
\end{align*}
Subsequently, let $\widehat X=\widehat V^D$. We verify that $\widehat X$ satisfies the defining equations of the dual Drazin inverse in \eqref{Def Dual Drazin inverse}. We first verify the commutativity condition $\widehat V\widehat X=\widehat X\widehat V$. A direct computation gives 
\begin{align}\label{VX}
    \widehat V\widehat X=\begin{pmatrix}
        \widehat B^e&\widehat A\widehat A^\pi \widehat B^e-\widehat B\widehat A\widehat A^\pi \widehat B^D\\O&\widehat B^e
    \end{pmatrix}=\begin{pmatrix}
        \widehat B^e&O\\O&\widehat B^e
    \end{pmatrix}.
\end{align}
Similarly,
\begin{align*}
    \widehat X\widehat V=\begin{pmatrix}
        \widehat B^e&O\\\widehat B^D\widehat A\widehat A^\pi-\widehat A\widehat A^\pi \widehat B^D&\widehat B^e
    \end{pmatrix}=\begin{pmatrix}
        \widehat B^e&O\\O&\widehat B^e
    \end{pmatrix}.
\end{align*}
Since $\widehat A\widehat A^\pi$ commutes with $\widehat B$ and hence with $\widehat B^D$, the $(2,1)$-block is zero. Therefore, $\widehat V\widehat X=\widehat X\widehat V$.

Then, from \eqref{VX}, we obtain
\begin{align*}
    \widehat X\widehat V\widehat X=\begin{pmatrix}
        O&\widehat B^e\\\widehat B^D&-\widehat A\widehat A^\pi \widehat B^D
    \end{pmatrix}\begin{pmatrix}
        \widehat B^e&O\\O&\widehat B^e
    \end{pmatrix}=\begin{pmatrix}
        O&\widehat B^e\\\widehat B^D&-\widehat A\widehat A^\pi \widehat B^D
    \end{pmatrix}=\widehat X.
\end{align*}
Note that $\widehat A^\pi$ is the spectral idempotent associated with the zero eigenvalue of $\widehat A$, it follows that $(\widehat A\widehat A^\pi)^l=O$ for all $l\geq i_{\widehat A}$. Moreover, there exists a positive integer number $k$ such that 
\[\widehat V^k\widehat V^\pi=
\left\{
\begin{aligned}
    &\begin{pmatrix}
        O&\widehat B^{\frac{k+1}{2}}\widehat B^\pi\\\widehat B^{\frac{k-1}{2}}\widehat B^\pi&O
    \end{pmatrix},~\text{k is odd;}\\
    &\begin{pmatrix}
        \widehat B^{\frac{k}{2}}\widehat B^\pi&O\\O&\widehat B^{\frac{k}{2}}\widehat B^\pi
    \end{pmatrix}~\text{k is even.}
\end{aligned}
\right.
\] 

Hence, the non-nilpotent part of $\widehat V$ is governed by $\widehat B$, whose index is finite. Consequently, there exists a positive integer $k$ such that $\widehat V^k\widehat V^\pi=O,$ i.e., $\widehat V^{k+1}\widehat X=\widehat V^k$. This completes the verification.

Furthermore, under the condition $\widehat V\widehat A\widehat A^e\widehat B=O$, we have $\widehat U\widehat V=O$. Define $\alpha=0$ for odd $i$ and $\alpha=1$ for even $i$. By Lemma \ref{PQ=0}, we conclude that
\begin{align*}
   \widehat {N}^D&=(\widehat U+\widehat V)^D=\sum_{i=0}^{i_{\widehat V}-1}\widehat V^i\widehat V^\pi \widehat U^{(i+1)D}+\widehat V^D\widehat U^\pi
    \\&=\begin{pmatrix}
        \alpha\sum\limits_{i=0}^{i_{\widehat V}-1}\widehat B^\pi \widehat B^{\frac{i}{2}}\widehat A^{(i+1)^D}&O
        \\(1-\alpha)\sum\limits_{i=0}^{i_{\widehat V}-1}\widehat B^\pi \widehat B^{[\frac{i-1}{2}]}\widehat A^{(i+1)D}&O
    \end{pmatrix}+\begin{pmatrix}
        O&\widehat B\widehat B^D\\\widehat B^D\widehat A^\pi &-\widehat A\widehat A^\pi \widehat B^D
    \end{pmatrix}
    \\&=\begin{pmatrix}
     \sum\limits_{i=0}^{i_{\widehat B}-1}\widehat B^\pi \widehat B^i\widehat A^{(2i+1)D}&\widehat B\widehat B^D\\
     \sum\limits_{i=0}^{i_{\widehat B}-1}\widehat B^\pi \widehat B^i\widehat A^{(2i+2)D}+\widehat B^D\widehat A^\pi&-\widehat A\widehat A^\pi \widehat B^D
 \end{pmatrix},
\end{align*}
where the identity $\widehat A\widehat A^\pi \widehat B^\pi \widehat A^D=\widehat A\widehat A^\pi (I-\widehat B^e)\widehat A^D=-\widehat A\widehat A^\pi \widehat B^e\widehat A^D=-\widehat B^e\widehat A\widehat A^\pi \widehat A^D=O$ is instrumental in the above computation.

This completes the proof. 
\end{proof}

The following theorem is the symmetric counterpart of Theorem \ref{DecompABIO1}, and the proof is omitted. 
\begin{theorem}\label{DecompABIO2}
   Let $\widehat N=\begin{pmatrix}
    \widehat A&\widehat B\\I&O
\end{pmatrix}\in\mathbb{DC}_z^{n\times n}$ be an anti-triangular block matrix, where $\widehat A$ and $\widehat B$ are of the same order and dual Drazin invertible. Assume that $\Ind_s(\widehat A)=i_{\widehat A}$, $\Ind_s(\widehat B)=i_{\widehat B}$ and the conditions $\widehat A\widehat A^\pi \widehat B=\widehat B\widehat A\widehat A^\pi$ and $\widehat B\widehat A\widehat A^e=O$ hold. Then
\begin{align*}
 \widehat N^D=\begin{pmatrix}
     \sum\limits_{i=0}^{i_{\widehat B}-1}\widehat A^{(2i+1)D}\widehat B^\pi \widehat B^i
     &\widehat A^\pi \widehat B\widehat B^D+\sum\limits_{i=0}^{i_{\widehat B}-1}\widehat A^{(2i+2)D}\widehat B^\pi \widehat B^{i+1}
     \\\sum\limits_{i=0}^{i_{\widehat B}-1}\widehat A^{(2i+2)D}\widehat B^\pi \widehat B^i+\widehat A^\pi \widehat B^D&-\widehat A\widehat A^\pi \widehat B^D-\widehat A^D\widehat B^e
     +\sum\limits_{i=0}^{i_{\widehat B}-1}\widehat A^{(2i+3)D}\widehat B^\pi \widehat B^{i+1}
 \end{pmatrix}.
\end{align*}
\end{theorem}

To obtain the explicit formulae for the dual Drazin inverse of the anti-triangular block matrix $\widehat M=\begin{pmatrix}
    \widehat A&\widehat B\\\widehat C&O
\end{pmatrix}$ under new assumptions, we apply Theorem \ref{DecompABIO1} and Theorem \ref{DecompABIO2}.
\begin{theorem}\label{DecompABCO1}
   Let $\widehat M=\begin{pmatrix}
    \widehat A&\widehat B\\\widehat C&O
\end{pmatrix}\in\mathbb{DC}_z^{n\times n}$ be an anti-triangular block matrix, where $\widehat A$ and $\widehat B\widehat C$ are of the same order and dual Drazin invertible. Assume that $\Ind_s(\widehat A)=i_{\widehat A}$, $\Ind_s(\widehat B\widehat C)=i_{\widehat B\widehat C}$ and the conditions $\widehat A\widehat A^\pi \widehat B\widehat C=\widehat B\widehat C\widehat A\widehat A^\pi$ and $\widehat A\widehat A^e\widehat B\widehat C=O$ hold. Then
\begin{align}\label{E_i}
 \widehat M^D=\begin{pmatrix}
     \sum\limits_{i=0}^{i_{\widehat B\widehat C}-1}(\widehat B\widehat C)^\pi (\widehat B\widehat C)^i\widehat A^{(2i+1)D}
     &\sum\limits_{i=0}^{i_{\widehat B\widehat C}-1}(\widehat B\widehat C)^\pi (\widehat B\widehat C)^{i}\widehat A^{(2i+2)D}\widehat B+(\widehat B\widehat C)^D\widehat A^\pi \widehat B\\
     \\\sum\limits_{i=0}^{i_{\widehat B\widehat C}-1}\widehat C(\widehat B\widehat C)^\pi (\widehat B\widehat C)^i\widehat A^{(2i+2)D}&+\sum\limits_{i=0}^{i_{\widehat B\widehat C}-1}\widehat C(\widehat B\widehat C)^\pi (\widehat B\widehat C)^{i}\widehat A^{(2i+3)D}\widehat B
     \\+\widehat C(\widehat B\widehat C)^D\widehat A^\pi&-\widehat C(\widehat B\widehat C)^{2D}\widehat A\widehat A^\pi \widehat B-\widehat C(\widehat B\widehat C)^D\widehat A^D\widehat B
 \end{pmatrix}.
\end{align}
\end{theorem}

\begin{proof}
The anti-triangular matrix $\widehat M$ can be decomposed as follows:   
\begin{align*}
    \widehat M=\begin{pmatrix}
        I&O\\O&\widehat C
    \end{pmatrix}\begin{pmatrix}
        \widehat A&\widehat B\\I&O
    \end{pmatrix}:=\widehat H\widehat K.
\end{align*}
Utilizing Lemma \ref{cline}, we obtain
\begin{align}\label{Mcline}
    \widehat M^D=\widehat H(\widehat K\widehat H)^{2D}\widehat K=\begin{pmatrix}
        I&O\\O&\widehat C
    \end{pmatrix}\begin{pmatrix}
        \widehat A&\widehat B\widehat C\\I&O
    \end{pmatrix}^{2D}\begin{pmatrix}
        \widehat A&\widehat B\\I&O
    \end{pmatrix}.
\end{align}
Furthermore, under the conditions  $\widehat A\widehat A^\pi \widehat B\widehat C=\widehat B\widehat C\widehat A\widehat A^\pi$ and $\widehat AA^e\widehat B\widehat C=O$, by applying Theorem \ref{DecompABIO1}, we get 
\begin{align*}
    (\widehat K\widehat H)^D=\begin{pmatrix}
     \sum\limits_{i=0}^{i_{\widehat B\widehat C}-1}(\widehat B\widehat C)^\pi (\widehat B\widehat C)^i\widehat A^{(2i+1)D}&(\widehat B\widehat C)^e\\
     \sum\limits_{i=0}^{i_{\widehat B\widehat C}-1}(\widehat B\widehat C)^\pi (\widehat B\widehat C)^i\widehat A^{(2i+2)D}+(\widehat B\widehat C)^D\widehat A^\pi&-\widehat A\widehat A^\pi (\widehat B\widehat C)^D
 \end{pmatrix}
\end{align*}
By routine calculations, we derive
\begin{align*}
   \begin{pmatrix}
        I&O\\O&\widehat C
    \end{pmatrix}\begin{pmatrix}
        \widehat A&\widehat B\widehat C\\I&O
    \end{pmatrix}^{D}=\begin{pmatrix}
     \sum\limits_{i=0}^{i_{\widehat B\widehat C}-1}(\widehat B\widehat C)^\pi (\widehat B\widehat C)^i\widehat A^{(2i+1)D}&(\widehat B\widehat C)^e\\
     \sum\limits_{i=0}^{i_{\widehat B\widehat C}-1}\widehat C(\widehat B\widehat C)^\pi (\widehat B\widehat C)^i\widehat A^{(2i+2)D}+\widehat C(\widehat B\widehat C)^D\widehat A^\pi&-\widehat C\widehat A\widehat A^\pi (\widehat B\widehat C)^D
 \end{pmatrix}, 
\end{align*}and
\begin{align*}
    \begin{pmatrix}
        \widehat A&\widehat B\widehat C\\I&O
    \end{pmatrix}^{D}\begin{pmatrix}
        \widehat A&\widehat B\\I&O
    \end{pmatrix}=\begin{pmatrix}
     \sum\limits_{i=0}^{i_{\widehat B\widehat C}-1}(\widehat B\widehat C)^\pi (\widehat B\widehat C)^i\widehat A^{(2i+1)D}\widehat A+(\widehat B\widehat C)^e&
     \sum\limits_{i=0}^{i_{\widehat B\widehat C}-1}(\widehat B\widehat C)^\pi (\widehat B\widehat C)^i\widehat A^{(2i+1)D}\widehat B\\
     \\\sum\limits_{i=0}^{i_{\widehat B\widehat C}-1}(\widehat B\widehat C)^\pi (\widehat B\widehat C)^i\widehat A^{(2i+1)D}&\sum\limits_{i=0}^{i_{\widehat B\widehat C}-1}(\widehat B\widehat C)^\pi (\widehat B\widehat C)^i\widehat A^{(2i+2)D}\widehat B\\+(\widehat B\widehat C)^D\widehat A^\pi \widehat A-\widehat A\widehat A^\pi (\widehat B\widehat C)^D&+(\widehat B\widehat C)^D\widehat A^\pi \widehat B
 \end{pmatrix}.
\end{align*}
Furthermore, by substituting the above representations into \eqref{Mcline}, we obtain
{\small
\begin{align*}
    \widehat M^D=\begin{pmatrix}
         \sum\limits_{i=0}^{i_{\widehat B\widehat C}-1}(\widehat B\widehat C)^\pi (\widehat B\widehat C)^i\widehat A^{(2i+1)D}+(\widehat B\widehat C)^D\widehat A\widehat A^\pi
         & \sum\limits_{i=0}^{i_{\widehat B\widehat C}-1}(\widehat B\widehat C)^\pi (\widehat B\widehat C)^i\widehat A^{(2i+2)D}\widehat B
         \\-\widehat A\widehat A^\pi (\widehat B\widehat C)^D&+(\widehat B\widehat C)^D\widehat A^\pi \widehat B\\
         \\\sum\limits_{i=0}^{i_{\widehat B\widehat C}-1}\widehat C(\widehat B\widehat C)^\pi (\widehat B\widehat C)^i\widehat A^{(2i+2)D}+\widehat C(\widehat B\widehat C)^D\widehat A^\pi(\widehat B\widehat C)^e
         &\sum\limits_{i=0}^{i_{\widehat B\widehat C}-1}\widehat C(\widehat B\widehat C)^\pi (\widehat B\widehat C)^i\widehat A^{(2i+3)D}\widehat B
         \\+\sum\limits_{i=0}^{i_{\widehat B\widehat C}-1}\widehat C(\widehat B\widehat C)^D\widehat A^\pi
         (\widehat B\widehat C)^\pi (\widehat B\widehat C)^i\widehat A^{(2i+1)D}\widehat A
         &+\sum\limits_{i=0}^{i_{\widehat B\widehat C}-1}\widehat C(\widehat B\widehat C)^D\widehat A^\pi
         (\widehat B\widehat C)^\pi (\widehat B\widehat C)^i\widehat A^{(2i+1)D}\widehat B\\-\widehat C\widehat A\widehat A^\pi(\widehat B\widehat C)^{2D}\widehat A^\pi \widehat A+\widehat C\widehat A\widehat A^\pi(\widehat B\widehat C)^{D}\widehat A^\pi \widehat A(\widehat B\widehat C)^D&-\widehat C\widehat A\widehat A^\pi(\widehat B\widehat C)^{2D}\widehat A^\pi \widehat B
    \end{pmatrix}.
\end{align*}
}
After straightforward simplifications, the proof is complete. 
\end{proof}


In 2009, Catral et al. \cite{Catral2009} introduced a representation of the Drazin inverse of a bipartite matrix over $\mathbb{C}$ with block form $\widehat {\widetilde{M}}$, which is a direct consequence of Corollary \ref{OBCO}.

\begin{corollary}\label{OBCO}
     Let $\widehat {\widetilde{M}}=\begin{pmatrix}
    O&\widehat B\\\widehat C&O
\end{pmatrix}\in\mathbb{DC}_z^{n\times n}$ be an anti-triangular block matrix, where the square matrix $\widehat B\widehat C$ is dual Drazin invertible. Then
\begin{align*}
    \widehat {\widetilde{M}}^D=\begin{pmatrix}
      O&(\widehat B\widehat C)^{D}\widehat B\\\widehat C(\widehat B\widehat C)^{D}&O
    \end{pmatrix}.
\end{align*}
\end{corollary}

The following theorem can be proved analogously to Theorem \ref{DecompABCO1}, and the proof is omitted.
\begin{theorem}\label{DecompABCO2}
  Let $\widehat M=\begin{pmatrix}
    \widehat A&\widehat B\\\widehat C&O
\end{pmatrix}\in\mathbb{DC}_z^{n\times n}$ be an anti-triangular block matrix, where $\widehat A$ and $\widehat B\widehat C$ are of the same order and dual Drazin invertible. Assume that $\Ind_s(\widehat A)=i_{\widehat A}$, $\Ind_s(\widehat B\widehat C)=i_{\widehat B\widehat C}$ and the conditions $\widehat A\widehat A^\pi \widehat B\widehat C=\widehat B\widehat C\widehat A\widehat A^\pi$ and $\widehat B\widehat C\widehat A\widehat A^e=O$ hold. Then
\begin{align*}
\widehat M^D=\begin{pmatrix}
     \sum\limits_{i=0}^{i_{\widehat B\widehat C}-1}\widehat A^{(2i+2)D}(\widehat B\widehat C)^\pi (\widehat B\widehat C)^i\widehat A+\widehat A^D(\widehat B\widehat C)^e
     &\sum\limits_{i=0}^{i_{\widehat B\widehat C}-1}\widehat A^{(2i+2)D}(\widehat B\widehat C)^\pi (\widehat B\widehat C)^{i}\widehat B
     \\+\sum\limits_{i=0}^{i_{\widehat B\widehat C}-1}\widehat A^{(2i+3)D}(\widehat B\widehat C)^\pi (\widehat B\widehat C)^{i+1}&+\widehat A^\pi(\widehat B\widehat C)^D\widehat B\\
     \\\sum\limits_{i=0}^{i_{\widehat B\widehat C}-1}\widehat C\widehat A^{(2i+3)D}(\widehat B\widehat C)^\pi (\widehat B\widehat C)^i\widehat A-\widehat C\widehat A^{2D}(\widehat B\widehat C)^e
     &\sum\limits_{i=0}^{i_{\widehat B\widehat C}-1}\widehat C\widehat A^{(2i+3)D}(\widehat B\widehat C)^\pi (\widehat B\widehat C)^{i}\widehat B
     \\+\sum\limits_{i=0}^{i_{\widehat B\widehat C}-1}\widehat C\widehat A^{(2i+4)D}(\widehat B\widehat C)^\pi (\widehat B\widehat C)^{i+1}+\widehat C\widehat A^\pi(\widehat B\widehat C)^D&-\widehat C\widehat A\widehat A^\pi(\widehat B\widehat C)^{2D}\widehat B-\widehat C\widehat A^D(\widehat B\widehat C)^D\widehat B
 \end{pmatrix}.
\end{align*}
\end{theorem}

\section{Applications to Certain Graph Classes}
In this section, we introduce dual-number-weighted double star digraphs, dual-number-weighted D-linked star digraphs, and dual-number-weighted Dutch windmill digraphs. Then we investigate their structures and the associated adjacency matrices, derive representations of the dual Drazin inverse, and generalize several existing results in the literature. 

\begin{definition}
Let $DG^w=(V, E, w)$ be a connected dual-number-weighted digraph, where $V=(v_1,...,v_n)$ is the vertex
set, the arc set is defined by
\[
E=\{\, v_{ij}\mid v_{ij} \text{ is the directed arc from } v_i \text{ to } v_j \,\},
\]
and $w:E\to\mathbb{DC}$ assigns a dual-number weight to each arc.
\end{definition}
The adjacency matrix of $DG^w$ is defined as $\widehat{A}=(\widehat{a}_{ij})$, where
\[
\widehat{a}_{ij}=
\begin{cases}
w(v_{ij}), & \text{if } v_{ij}\in E,\\
0, & \text{otherwise},
\end{cases}
\]
 and where we also assume that  $\widehat{A}\in\mathbb{DC}_z^{n\times n}.$

\subsection{Dual-Number-Weighted Double Star Digraphs (DN-DS digraphs)}
The double star digraph $S_{m,n}$ is constructed from two star digraphs
$K_{1,m-1}$ and $K_{1,n-1}$ by joining their central vertices with two
directed edges in opposite directions. The resulting digraph preserves
the star structure on both sides while introducing a bidirectional
connection between the two centers. We illustrate the structure of the double star digraph in Figure~\ref{double star digraph}.
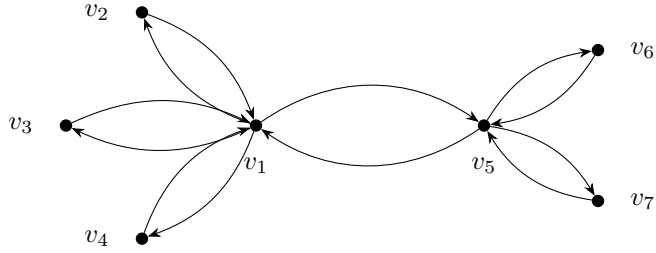
\begin{figure}[htbp]
\centering
\begin{tikzpicture}[
   vertex/.style={circle, draw, fill=black, inner sep=1.5pt},
  every label/.style={font=\small},
  label distance=6pt,
  >=Stealth
]

\node[vertex,label=below:$v_1$] (v1) at (-1.5,0) {};
\node[vertex,label=below:$v_5$] (v2) at ( 1.5,0) {};

\node[vertex,label=left:$v_2$] (v3) at (-3, 1.5) {};
\node[vertex,label=left:$v_3$] (v4) at (-4, 0) {};
\node[vertex,label=left:$v_4$] (v5) at (-3,-1.5) {};

\node[vertex,label=right:$v_6$] (v6) at ( 3, 1) {};
\node[vertex,label=right:$v_7$] (v7) at ( 3,-1) {};


\draw[->, bend left=35]  (v1) to (v2);   
\draw[->, bend right=-35] (v2) to (v1);   

\foreach \i in {3,4,5}{
  \draw[->, bend left=25]  (v1) to (v\i); 
  \draw[->, bend right=-25] (v\i) to (v1); 
}

\foreach \i in {6,7}{
  \draw[->, bend left=25]  (v2) to (v\i); 
  \draw[->, bend right=-25] (v\i) to (v2); 
}

\end{tikzpicture}
\caption{Double star digraph $S_{4,3}$.}
\label{double star digraph}
\end{figure}

We consider the double star digraph $S_{m,n}$ with dual-number weights assigned to its arcs, referred to as the dual-number weighted double star digraph, abbreviated as DN-DS digraph, and denoted by $DWS_{m,n}$.  The corresponding adjacency matrix $A(DWS_{m,n})$ is therefore a matrix over $\mathbb{DC}_z$. 

Let $\widehat{x},\widehat{y}\in\mathbb{DC}_z^m$ and $\widehat{\omega},\widehat{v}\in\mathbb{DC}_z^n$ be nonzero vectors, and let $\widehat{a},\widehat{b}\in\mathbb{D}\setminus\{0\}$. Then any matrix whose DN-DS digraph is $WS_{(m+1),(n+1)}$ is permutation similar to $\widehat{M}$ of the following block structure:
\begin{align}\label{adj of wdsd}
\widehat{M}
=
\left(
\begin{array}{ccc: c}
0 & \widehat{x}^{T} & \widehat{a} & 0 \\
\widehat{y} & 0 & 0 & 0 \\ 
\widehat{b} & 0 & 0 & \widehat{\omega}^{T} \\\hdashline
0 & 0 & \widehat{v} & 0
\end{array}
\right)=\begin{pmatrix}
    \widehat{A}&\widehat{B}
    \\\widehat{C}&O
\end{pmatrix}\in\mathbb{DC}_z^{(m+n+2)\times (m+n+2)},
\end{align}
where $\widehat{A}=A+\varepsilon A_0$,
$\widehat{B}=B+\varepsilon B_0$, and 
$\widehat{C}=C+\varepsilon C_0$
are dual complex matrices.

Observe that the group inverse of the adjacency matrix of the double star digraph was obtained in \cite{MNSEJAA2022}, and its Drazin inverse subject to $x^Ty=w^Tv=0$ was studied in \cite{AMPMJM2026}. We proceed by weakening the assumptions in \cite{AMPMJM2026} and by lifting the problem from $\mathbb{C}$ to the dual complex algebra $\mathbb{DC}_z$. This leads to explicit dual Drazin inverse representation for DN-DS digraphs.

Subsequently, we present the following theorem.

\begin{theorem}
Let $\widehat{M}$ be a dual complex matrix given by \eqref{adj of wdsd}, associated with $DWS_{(m+1), (n+1)}$. Assume that $\omega^{\,T}{v}+\varepsilon(w^Tv_0+w_0^Tv)=0$ and $\operatorname{Ind_s}(\widehat{A})=k$, then 
\begin{align*}
\widehat{M}^{D}
&=
\begin{pmatrix}
0 & \theta^{D} x^{T} & \theta^{D} a & 0 \\
y \theta^{D} & 0 & 0 & y \theta^{2D} a \omega^{T} \\
b \theta^{D} & 0 & 0 & b \theta^{2D} a \omega^{T} \\
0 & v b \theta^{2D} x^{T} & v b \theta^{2D} a & 0
\end{pmatrix}
+ \varepsilon
\begin{pmatrix}
0 & \theta^{D} x_0^{T} + \theta_R x^{T}
& \theta^{D} a_0 + \theta_R a & 0 \\[2mm]
y \theta_R + y_0 \theta^{D} & 0 & 0 & M_0^{24} \\[2mm]
b \theta_R + b_0 \theta^{D} & 0 & 0 & M_0^{34} \\[2mm]
0 & M_0^{42} & M_0^{43} & 0
\end{pmatrix},
\end{align*}
where
\[
\begin{aligned}
M_0^{24}
&= y \theta^{2D} a \omega_0^{T}
+ y \theta^{2D} a_0 \omega^{T}
+ y \theta^{D}\theta_R a \omega^{T}
+ y\theta_R \theta^{D} a \omega^{T}
+ y_0 \theta^{2D} a \omega^{T}, \\[1mm]
M_0^{34}
&= b \theta^{2D} a \omega_0^{T}
+ b \theta^{2D} a_0 \omega^{T}
+ b \theta^{D} \theta_{R} a \omega^{T}
+ b \theta_{R}\theta^{D} a \omega^{T}
+ b_0 \theta^{2D} a \omega^{T}, \\[1mm]
M_0^{42}
&= v b \theta^{2D} x_0^{T}
+ v b \theta^{D} \theta_Rx^{T}
+ v b \theta_R\theta^{D} x^{T}
+ v b_0 \theta^{2D} x^{T}
+ v_0 b \theta^{2D} x^{T}, \\[1mm]
M_0^{43}
&= v b \theta^{2D} a_0
+ v b \theta^{D}\theta_R a
+ v b \theta_R\theta^{D} a
+ v b_0 \theta^{2D} a
+ v_0 b \theta^{2D} a.
\end{aligned}
\]
\end{theorem}
\begin{proof}
Note that $\omega^{\,T}{v}+\varepsilon(w^Tv_0+w_0^Tv)=0$, i.e., $\widehat{B}\widehat{C}=O$. By Theorem \ref{DecompABCO1}, we obtain
\begin{align}\label{PFM^D}
    \widehat{M}^D=\begin{pmatrix}
        \widehat{A}^D&\widehat{A}^{2D}\widehat{B}\\
        \widehat{C}\widehat{A}^{2D}
        &\widehat{C}\widehat{A}^{3D}\widehat{B}
    \end{pmatrix}
\end{align}
Moreover, let $\widehat{\theta}=\widehat{x}^{\,T}\widehat{y}+\widehat{a}\widehat{b}=\theta+\varepsilon \theta_0$. Then apply Corollary \ref{OBCO} and \eqref{adj of wdsd}, we have
\begin{align}\label{DSDA^D}
    \widehat{A}^D=\begin{pmatrix}
        0&\widehat{\theta}^D\widehat{x}^T&\widehat{\theta}^D\widehat{a}\\
        \widehat{y}\widehat{\theta}^D&0&0\\
        \widehat{b}\widehat{\theta}^D&0&0\\
        \end{pmatrix}.
\end{align}
From the properties of the dual matrix, we note that
\[
\left\{
\begin{aligned}
    &\widehat{\theta}^D
    =\theta^D+\varepsilon \theta_R=(\widehat{x}^T\widehat{y}+\widehat{a}\widehat{b})^{-1};\\
    &\theta_0=x^Ty_0+x_0^Ty+ab_0+a_0b;\\
    &\theta_R=-\theta^D\theta_0\theta^D
    +\sum_{i=0}^{k-1}(\theta^D)^{i+2}\theta_0\theta^i(I_n-\theta\theta^D)+\sum_{i=0}^{k-1}(I_n-\theta\theta^D)\theta^i\theta_0(\theta^D)^{i+2}.
\end{aligned}
\right.
\] 
Substituting \eqref{DSDA^D} and the above expressions into \eqref{PFM^D} completes the proof.
\end{proof}

\subsection{Dual-Number-Weighted D-Linked Stars Digraphs (DN-DLS digraphs)}
We consider a class of structured digraphs constructed from directed star components and a base digraph. 
Let $D$ be a digraph with vertex set $\{v_1,\ldots,v_n\}$, and for each $i$ let 
$K_{1,r_i}$ denote a directed star with $r_i$ pendant vertices. 
A $D$-linked stars digraph is formed by assembling the stars 
$K_{1,r_1},\ldots,K_{1,r_n}$ and introducing interconnections between their center vertices according to the adjacency relations in $D$. Specifically, whenever $(v_i,v_j)$ is an edge of $D$, a directed edge is added from the center of $K_{1,r_i}$ to the center of $K_{1,r_j}$. 
This construction produces a composite digraph whose global structure is governed by the base digraph $D$, while the local structure is determined by the individual 
star components.  The $D$-linked stars digraph, denoted by $gls(D, r_1,r_2,...,r_n)$, is shown in Figure~\ref{DN-DLS}.

\begin{figure}[htbp]
\centering
\begin{tikzpicture}[
  vertex/.style={circle, draw, fill=black, inner sep=1.5pt},
  >=Stealth,
  every label/.style={font=\small},
  label distance=6pt
]

\node[vertex,label=below:$v_1$] (c1) at (-1.5,-2) {};
\node[vertex,label=below:$v_2$] (c2) at (0,0) {};
\node[vertex,label=below:$v_3$] (c3) at (1.5,-2) {};

\node[vertex,label=left:$v_4$] (v4) at (-3, -1.2) {};
\node[vertex,label=left:$v_5$] (v5) at (-3, -2.8) {};

\node[vertex,label=above:$v_6$] (v6) at (-1.5, 1.2) {};
\node[vertex,label=above:$v_7$] (v7) at ( 0.0, 1.6) {};
\node[vertex,label=above:$v_8$] (v8) at ( 1.5, 1.2) {};

\node[vertex,label=right:$v_9$]  (v9)  at (3, -1.2) {};
\node[vertex,label=right:$v_{10}$] (v10) at (3,-2.8) {};


\draw[->, bend left=25]  (c1) to (v4);
\draw[->, bend right=-25] (v4) to (c1);
\draw[->, bend left=25]  (c1) to (v5);
\draw[->, bend right=-25] (v5) to (c1);

\draw[->, bend left=20]  (c2) to (v6);
\draw[->, bend right=-20] (v6) to (c2);
\draw[->, bend left=20]  (c2) to (v7);
\draw[->, bend right=-20] (v7) to (c2);
\draw[->, bend left=20]  (c2) to (v8);
\draw[->, bend right=-20] (v8) to (c2);

\draw[->, bend left=25]  (c3) to (v9);
\draw[->, bend right=-25] (v9) to (c3);
\draw[->, bend left=25]  (c3) to (v10);
\draw[->, bend right=-25] (v10) to (c3);


\draw[->, bend left=20]  (c1) to (c2);
\draw[->, bend right=-20] (c2) to (c1);

\draw[->, bend left=20]  (c2) to (c3);
\draw[->, bend right=-20] (c3) to (c2);

\draw[->, bend left=20]  (c1) to (c3);
\draw[->, bend right=-20] (c3) to (c1);

\end{tikzpicture}
\caption{$D$-linked stars digraph $\mathrm{gls}(D;2,3,2)$.}
\label{DN-DLS}
\end{figure}
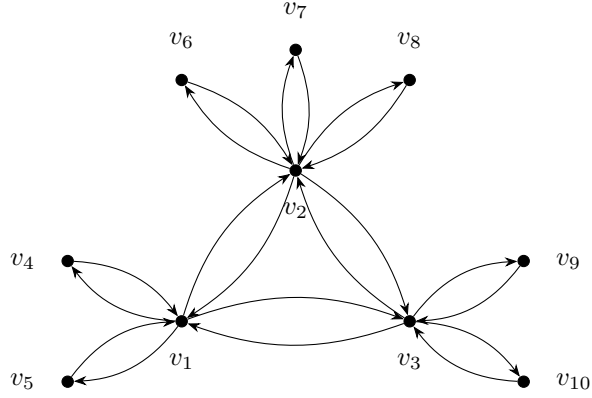

Such digraphs naturally extend double star digraphs and may be viewed as a 
particular type of composite or cluster-based network \cite{ZYLDAM2009} in which star subgraphs 
serve as building blocks. This modular construction is convenient for analyzing 
structural properties of associated matrices, including generalized inverses.

To incorporate algebraic weights, we assign dual-number weights to the edges of 
the digraph. The resulting weighted digraph is called a 
dual-number-weighted $D$-linked stars digraph (DN-DLS digraph), denoted by $DWgls(D, r_1,r_2,...,r_n)$. 
In this setting, the adjacency matrix becomes a dual-number matrix whose block 
structure reflects both the star components and the linking pattern induced by $D$.

Let $\widehat{x_i},\widehat{y_i} \in \mathbb{DC}_z^{r_i}$ be nonzero vectors, 
$r_i \in \mathbb{N}$ for $i=1,\ldots,n$, and let 
$\widehat{A} \in \mathbb{DC}_z^{n \times n}$. 
Then any adjacency matrix associated with the DN-DLS digraph 
$DWgls(D, r_1, \ldots, r_n)$ is permutation similar to $\widehat{M}$ of the following block form:
\begin{align}\label{DN-DLSM}
    \widehat{M}=\begin{pmatrix}
        \widehat{A}&\widehat{B}\\\widehat{C}&O
    \end{pmatrix}=\begin{pmatrix}
        A+\varepsilon A_0&B+\varepsilon B_0\\C+\varepsilon C_0&O
    \end{pmatrix}\in\mathbb{DC}_z^{\big(n+\sum\limits_{i=1}^{n}r_i\big)\times \big(n+\sum\limits_{i=1}^{n}r_i\big)},
\end{align}
where $\widehat{A}=A+\varepsilon A_0$,
\begin{align*}
&\widehat{B}=B+\varepsilon B_0=diag(\widehat{x_1}^T,...,\widehat{x_n}^T)=diag(x_1^T+\varepsilon \widetilde{x}_1^T,...,
x_n^T+\varepsilon \widetilde{x}_n^T),\\ &\widehat{C}=C+\varepsilon C_0=diag(\widehat{y_1},...,\widehat{y_n})=diag(y_1+\varepsilon \widetilde{y}_1,...,y_n+\varepsilon \widetilde{y}_n).
\end{align*}

It was shown in \cite{MNSEJAA2022} that the matrix $\widehat{M}$ is singular 
and its group inverse can be computed explicitly. 
However, Ara\'{u}jo et al. \cite{AMPMJM2026} determined the Drazin index only in the case $BC=0$, and consequently raised the problem of whether $M^d$ admits a constructive expression in terms of $A^D$ and the off-diagonal blocks $B$ and $C$ under the condition $BC=0$. 

To address this problem, we derive a representation of the Drazin inverse of the adjacency matrix of the DN-DLS digraph, extending the corresponding results to $\mathbb{DC}_z$.

\begin{theorem}
Let $\widehat{M}$ be a dual complex matrix given by \eqref{DN-DLSM}, associated with a DN-DLS digraph
$DWgls(D,r_1,r_2,...,r_n)$. If $x_i^Ty_i+\varepsilon(x_i^T\widetilde{y}_i+\widetilde{x}_i^Ty_i)=0$ and $\operatorname{Ind_s}(\widehat{A})=k$, then $\widehat{M}^{D}
= M^{D} + \varepsilon M_0$
\begin{align*}
    &=
\begin{pmatrix}
A^{D} & A^{2D}B \\
CA^{2D} & CA^{3D}B
\end{pmatrix}
+
\varepsilon
\begin{pmatrix}
A_{R}
&
A^{2D}B_{0}+A^{D}A_{R}B+A_{R}A^{D}B
\\[6pt]
CA^{D}A_{R}+CA_{R}A^{D}+C_{0}A^{2D}
&
CA^{3D}B_{0}+CA^{2D}A_{R}B+C_{0}A^{3D}B\\ &+CA^{D}A_{R}A^{D}B+CA_{R}A^{2D}B
\end{pmatrix}.
\end{align*}
\end{theorem}
\begin{proof}
    Notice that $x_i^Ty_i+\varepsilon(x_i^T\widetilde{y}_i+\widetilde{x}_i^Ty_i)=0$, i.e., $\widehat{B}\widehat{C}=O$. By Theorem \ref{DecompABCO1}, it follows that
    \begin{align}\label{DN-DLSMD}
        \widehat{M}^D=\begin{pmatrix}
\widehat{A}^{D} & \widehat{A}^{2D}\widehat{B} \\
\widehat C\widehat A^{2D} & \widehat C\widehat A^{3D}\widehat B
\end{pmatrix}.
    \end{align}
    A direct computation shows that \[
\widehat{A}^D=A^D+\varepsilon A_R,\qquad
\widehat{A}^{2D}=A^{2D}+\varepsilon (A^DA_R+A_RA^D),
\]
and
\[
\widehat{A}^{3D}
=
A^{3D}
+\varepsilon\bigl(
A^{2D}A_R
+
A^DA_RA^D
+
A_RA^{2D}
\bigr).
\]

Substituting these expressions into \eqref{DN-DLSMD} completes the proof.
\end{proof}

\subsection{Dual-Number-Weighted Dutch Windmill Digraphs (DN-DW digraphs)}
The Dutch windmill graph $D_n^m$ is obtained by identifying one vertex from each of $m$ copies of the cycle graph $C_n$. 
The common vertex is called the hub, and each cycle is referred to as a blade.

More precisely, let $x$ denote the hub vertex. 
For $k=1,\dots,m$, let $D_n^k$ be the $k$-th cycle subgraph with vertex set
\[
\{x, v_{k1}, v_{k2}, \dots, v_{k,n-1}\}.
\]
Then
\[
V(D_n^m)
=
\{x\}\cup
\{v_{kj} : k=1,\dots,m,\; j=1,\dots,n-1\}.
\]

A special case occurs when $n=3$. 
In this case, $D_3^m$ is called the friendship graph, consisting of $m$ triangles sharing a common vertex. An example of the Dutch windmill digraph $D_4^4$ is shown in Figure~\ref{Dutch windmill}.

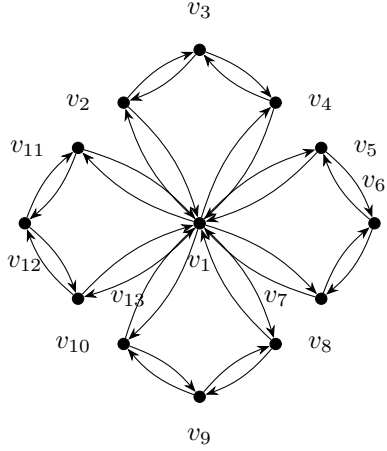
\begin{figure}[htbp]
\centering
\begin{tikzpicture}[
  vertex/.style={circle, draw, fill=black, inner sep=1.5pt},
  >=Stealth,
  every label/.style={font=\small},
  label distance=6pt
]

\newcommand{\bidiarc}[3]{%
  \draw[->, bend left=#3]  (#1) to (#2); 
  \draw[->, bend right=#3] (#2) to (#1); 
}

\node[vertex,label=below:$v_1$] (v1) at (0,0) {};

\node[vertex,label=left:$v_2$] (v2) at (-1, 1.6) {};
\node[vertex,label=above:$v_3$] (v3) at (0, 2.3) {};
\node[vertex,label=right:$v_4$] (v4) at (1, 1.6) {};

\node[vertex,label=right:$v_5$] (v5) at (1.6, 1) {};
\node[vertex,label=above:$v_6$] (v6) at (2.3, 0) {};
\node[vertex,label=left:$v_7$] (v7) at (1.6, -1) {};

\node[vertex,label=right:$v_8$] (v8) at (1,-1.6) {};
\node[vertex,label=below:$v_9$] (v9) at (0,-2.3) {};
\node[vertex,label=left:$v_{10}$] (v10) at (-1,-1.6) {};

\node[vertex,label=left:$v_{11}$] (v11) at (-1.6,1) {};
\node[vertex,label=below:$v_{12}$] (v12) at (-2.3,0) {};
\node[vertex,label=right:$v_{13}$] (v13) at (-1.6,-1) {};

\draw[->, bend left=15]  (v1) to (v2);
\draw[->, bend right=-15] (v2) to (v1);
\draw[->, bend left=15]  (v1) to (v4);
\draw[->, bend right=-15] (v4) to (v1);
\draw[->, bend right=-15] (v2) to (v3);
\draw[->, bend right=-15] (v3) to (v4);
\draw[->, bend right=-15] (v4) to (v3);
\draw[->, bend right=-15] (v3) to (v2);
\draw[->, bend left=15]  (v1) to (v5);
\draw[->, bend right=-15] (v5) to (v1);
\draw[->, bend left=15]  (v1) to (v7);
\draw[->, bend right=-15] (v7) to (v1);
\draw[->, bend right=-15] (v5) to (v6);
\draw[->, bend right=-15] (v6) to (v7);
\draw[->, bend right=-15] (v7) to (v6);
\draw[->, bend right=-15] (v6) to (v5);
\draw[->, bend left=15]  (v1) to (v8);
\draw[->, bend right=-15] (v8) to (v1);
\draw[->, bend left=15]  (v1) to (v10);
\draw[->, bend right=-15] (v10) to (v1);
\draw[->, bend right=-15] (v8) to (v9);
\draw[->, bend right=-15] (v9) to (v10);
\draw[->, bend right=-15] (v10) to (v9);
\draw[->, bend right=-15] (v9) to (v8);
\draw[->, bend left=15]  (v1) to (v11);
\draw[->, bend right=-15] (v11) to (v1);
\draw[->, bend left=15]  (v1) to (v13);
\draw[->, bend right=-15] (v13) to (v1);
\draw[->, bend left=15] (v11) to (v12);
\draw[->, bend left=15] (v12) to (v13);
\draw[->, bend right=-15] (v13) to (v12);
\draw[->, bend right=-15] (v12) to (v11);
\end{tikzpicture}
\caption{Dutch windmill digraph $D_4^{4}$.}
\label{Dutch windmill}
\end{figure}

It is known that $D_{2n}^m$ is bipartite, whereas $D_{2n+1}^m$ is non-bipartite. 
Consequently, the adjacency matrix of $D_{2n}^m$ is singular, while the adjacency matrix of $D_{2n+1}^m$ is non-singular \cite{Bapat2014}.

Let $A(D_n^m)$ denote the adjacency matrix of $D_n^m$. 
By ordering the vertices so that the hub vertex appears first and the vertices in each blade are grouped together, the adjacency matrix admits a block form reflecting the hub-and-blade decomposition of the graph. 
Moreover, since $D_{2n}^m$ is bipartite, a suitable reordering of the vertices according to the bipartition yields a bipartite block form of the adjacency matrix. 
These two block representations will be exploited in the sequel to derive explicit expressions for generalized inverses of the adjacency matrix. 

In this section, we consider a directed Dutch windmill graph with dual-number edge weights, called a dual-number-weighted Dutch windmill digraph (DN-DW digraph), denoted by $DWD_{2n}^m$. 
Let $\widehat{x_i},\widehat{y_i} \in \mathbb{DC}_z^{2n-1}$ be nonzero vectors, and let $\widehat{D_{2n}}^i \in \mathbb{DC}_z^{(2n-1) \times (2n-1)}$, for $i=1,...,m$. 
Then the adjacency matrix $\widehat{M}$ associated with the DN-DW digraph $DWD_{2n}^m$ has the following block form:
\begin{align}\label{Dutch-windmill}
\widehat{M}=\left(
 \begin{array}{c:ccc c}
0 & \widehat{x_1}^{T} & \widehat{x_2}^T & \cdots &\widehat{x_m}^T\\ \hdashline 
\widehat{y_1} & \widehat{D_{2n}}^1 & 0 & 0 &0\\
\widehat{y_2} & 0 & \widehat{D_{2n}}^2 & 0&0 \\
\vdots & \vdots & \vdots & \ddots&0\\\widehat{y_m}&0&0&0&\widehat{D_{2n}}^m
\end{array}
\right)=\begin{pmatrix}
    O&\widehat{B}
    \\\widehat{C}&\widehat{D}
\end{pmatrix}.
\end{align}
where $\widehat{D}=D+\varepsilon D_0$,
\begin{align*}
&\widehat{B}=B+\varepsilon B_0=\begin{pmatrix}
    \widehat{x_1}^T&\widehat{x_2}^T&\cdots&\widehat{x_m}^T
\end{pmatrix}=\begin{pmatrix}x_1^T+\varepsilon \widetilde{x}_1^T&\cdots&x_m^T+\varepsilon \widetilde{x}_m^T\end{pmatrix},\\ &\widehat{C}=C+\varepsilon C_0=\begin{pmatrix}
    \widehat{y_1}^T&\widehat{y_2}^T&\cdots&\widehat{y_m}^T
\end{pmatrix}=\begin{pmatrix}y_1^T+\varepsilon \widetilde{y}_1^T&\cdots&y_m^T+\varepsilon \widetilde{y}_m^T\end{pmatrix}.
\end{align*}

Meanwhile, the adjacency matrix \eqref{Dutch-windmill} can be transformed into a bipartite block form:
\begin{align}\label{bipartite}
    \widehat{M}'=\begin{pmatrix}
        O&\widehat{E}\\\widehat{F}&O
    \end{pmatrix}=\begin{pmatrix}
        O&E+\varepsilon E_0\\F+\varepsilon F_0&O
    \end{pmatrix}.
\end{align}

Notice that McDonald et al. \cite{MNSEJAA2022} derived the inverse of $A(D_{2n+1}^{m})$ and the group inverse of $A(D_{2n}^m)$ using the bipartite block representation. Furthermore, we derive representations of the dual Drazin inverse corresponding to the 
hub-and-blade form and bipartite block form, respectively.

\begin{theorem}\label{DWThm}
Let $\widehat{M}$ be a dual complex matrix given by \eqref{Dutch-windmill}, associated with a DN-DW digraph 
$DWD_{2n}^m$, $\kappa=2mn-m+1$. If $\widehat{D_{2n}}^s(\widehat{D_{2n}}^s)^e\widehat{y_s}\widehat{x_t}^T=0$ and $\widehat{D_{2n}}^s\widehat{y}_s\widehat{x}_t^T=\widehat{y}_s\widehat{x}_t^T\widehat{D_{2n}}^t(\widehat{D_{2n}}^t)^\pi$, for $s,t\in[1,m]$. Then
\begin{align*}
    \widehat{M}^D=\begin{pmatrix}
        \sum\limits_{i=0}^{i_{\phi}-1}B\phi^\pi \phi^iD^{(2i+3)D}C
        &\sum\limits_{i=0}^{i_{\phi}-1}B\phi^\pi \phi^iD^{(2i+2)D}+B\phi^DD^\pi
        \\-B\phi^{2D}DD^\pi C-B\phi^DD^DC& \\[6pt]\sum\limits_{i=0}^{i_{\phi}-1}\phi^\pi \phi^iD^{(2i+2)D}C+\phi^DD^\pi C
        &\sum\limits_{i=0}^{i_{\phi}-1}\phi^\pi \phi^iD^{(2i+1)D}
    \end{pmatrix}+\varepsilon \begin{pmatrix}
        \xi_1&\xi_2\\\xi_3&\xi_4
    \end{pmatrix},
\end{align*}where $\widehat\phi=\widehat{C}\widehat{B}=\phi+\varepsilon\phi_0$ and $\operatorname{Ind_s}(\widehat{C}\widehat{B})=i_\phi,$
\begin{align*}
    \xi_1&=\sum_{i=0}^{i_{\phi}-1}
    B\phi^\pi \phi^iD^{(2i+3)D}C_0
    +\sum_{i=0}^{i_{\phi}-1}
    \sum_{j=0}^{i-1}B\phi^\pi \phi^iD^{jD}D_RD^{(2i-j+2)D}C-\sum_{i=0}^{i_{\phi}-1}
    B\phi\phi_R\phi^iD^{(2i+3)D}C    \\&+\sum_{i=0}^{i_{\phi}-1}\sum_{j=1}^{i}B\phi^\pi \phi^{i-j}\phi_0\phi^{j-1}D^{(2i+3)D}C
    -\sum_{i=0}^{i_{\phi}-1}
    B\phi_0\phi^D\phi^iD^{(2i+3)D}C
    -\sum_{i=0}^{i_{\phi}-1}
    B_0\phi^\pi\phi^iD^{(2i+3)D}C
    \\&+B\phi^{2D}D^2D_RC+B\phi^{2D}DD_0D^DC-B\phi^{2D}D_0D^\pi C
    -B\phi^{2D}DD^\pi C_0-B\phi^{D}\phi_RDD^\pi C
    \\&+B\phi_R\phi^DDD^\pi C+B_0\phi^{2D}DD^\pi C
    -B\phi^DD^DC_0-B\phi^DD_RC-B\phi_RD^DC-B_0\phi^DD^DC-B_0\phi_RD^DC,
    \\[4pt]
    \xi_2&=\sum_{i=0}^{i_{\phi}-1}
    \sum_{j=0}^{i-1}B\phi^\pi \phi^iD^{jD}D_RD^{(2i-j+1)D}
    +\sum_{i=0}^{i_{\phi}-1}
    \sum_{j=1}^{i}B\phi^\pi \phi^{i-j}\phi_0\phi^{j-1}D^{(2i+2)D}-\sum_{i=0}^{i_{\phi}-1}
    B\phi\phi_R\phi^iD^{(2i+2)D}
    \\&-\sum_{i=0}^{i_{\phi}-1}
    B\phi_0\phi^D\phi^iD^{(2i+2)D}
    -\sum_{i=0}^{i_{\phi}-1}
    B_0\phi^\pi \phi^iD^{(2i+2)D}
    +B\phi^DDD_R+B\phi^DD_0D^D+B_0\phi^DD^\pi
    +B\phi_RD^\pi,
    \\[4pt]
    \xi_3&=\sum_{i=0}^{i_{\phi}-1}
    \phi^\pi \phi^iD^{(2i+2)D}C_0
    +\sum_{i=0}^{i_{\phi}-1}
    \sum_{j=0}^{i-1}\phi^\pi \phi^iD^{jD}
    D_RD^{(2i-j+1)D}C
    +\sum_{i=0}^{i_{\phi}-1}
    \sum_{j=1}^{i}\phi^\pi \phi^{i-j}\phi_0\phi^{j-1}D^{(2i+2)D}C
    \\&-\sum_{i=0}^{i_{\phi}-1}
    \phi\phi_R\phi^iD^{(2i+2)D}C-\sum_{i=0}^{i_{\phi}-1}
    \phi_0\phi^D \phi^iD^{(2i+2)D}C+\phi^DD^\pi C_0
    +\phi_RD^\pi C-\phi^DDD_RC-\phi^DD_0D^DC,
    \\[4pt]
    \xi_4&=\sum_{i=0}^{i_{\phi}-1}
    \sum_{j=0}^{i-1}\phi^\pi \phi^iD^{jD}D_RD^{(2i-j)D}
    +\sum_{i=0}^{i_{\phi}-1}
    \sum_{j=1}^{i}\phi^\pi \phi^{i-j}\phi_0\phi^{j-1}D^{(2i+1)D}
    -\sum_{i=0}^{i_{\phi}-1}
    \phi\phi_R \phi^iD^{(2i+1)D}\\
    &-\sum_{i=0}^{i_{\phi}-1}
    \phi_0\phi^D \phi^iD^{(2i+1)D}.
\end{align*}
\end{theorem}
\begin{proof}
    Let 
\[
P=\begin{pmatrix}O&I\\ I&O\end{pmatrix},
\]
which is a permutation matrix and hence orthogonal. In particular, $P^{-1}=P$.
A direct computation shows that
\[
P\begin{pmatrix}\widehat{A}&\widehat{B}\\ \widehat{C}&O\end{pmatrix}P
=
\begin{pmatrix}O&\widehat{C}\\ \widehat{B}&\widehat{A}\end{pmatrix}.
\]
Since the Drazin inverse is similarity invariant, i.e.,
\[
(SXS^{-1})^D = S X^D S^{-1}\quad \text{for any nonsingular } S,
\]
we obtain (taking $S=P$ and using $P^{-1}=P$) that
\begin{equation}\label{eq:perm-sim-Drazin}
\begin{pmatrix}0&\widehat{C}\\ \widehat{B}&\widehat{A}\end{pmatrix}^D
=
P\begin{pmatrix}\widehat{A}&\widehat{B}\\ \widehat{C}&O\end{pmatrix}^D P.
\end{equation}
By Theorem~\ref{DecompABCO1},
\[
\begin{pmatrix}\widehat{A}&\widehat{B}\\ \widehat{C}&O\end{pmatrix}^D
=
\begin{pmatrix}
\widehat{E}_1 & \widehat{E}_2\\
\widehat{E}_3 & \widehat{E}_4
\end{pmatrix},
\]
where $\widehat{E}_i$ $(i=1,2,3,4)$ are given in \eqref{E_i}. Substituting this into
\eqref{eq:perm-sim-Drazin} yields
\[
\begin{pmatrix}O&\widehat{C}\\ \widehat{B}&\widehat{A}\end{pmatrix}^D
=
\begin{pmatrix}
\widehat{E}_4 & \widehat{E}_3\\
\widehat{E}_2 & \widehat{E}_1
\end{pmatrix}.
\]
To be consistent with the adjacency-matrix ordering in \eqref{Dutch-windmill}, we identify
\[
\widehat{A}=\widehat{D},\qquad \widehat{C}=\widehat{B},\qquad \widehat{B}=\widehat{C},
\]
which corresponds to a relabeling induced by the vertex permutation used in \eqref{Dutch-windmill}.
With this identification, the preceding block formula gives
\begin{align*}
\widehat{M}^D=\begin{pmatrix}
\displaystyle \sum_{i=0}^{i_{\widehat{C}\widehat{B}}-1}
\widehat{B}(\widehat{C}\widehat{B})^\pi(\widehat{C}\widehat{B})^{i}\widehat{D}^{(2i+3)d}\widehat{C}
&
\displaystyle \sum_{i=0}^{i_{\widehat{C}\widehat{B}}-1}
\widehat{B}(\widehat{C}\widehat{B})^\pi(\widehat{C}\widehat{B})^{i}\widehat{D}^{(2i+2)d}+\widehat{B}(\widehat{C}\widehat{B})^{d}\widehat{D}^{\pi}\\-\widehat{B}(\widehat{C}\widehat{B})^{2d}\widehat{D}\widehat{D}^\pi \widehat{C}
-\widehat{B}(\widehat{C}\widehat{B})^{d}\widehat{D}^{d}\widehat{C}&
\\[8pt]
\displaystyle \sum_{i=0}^{i_{\widehat{C}\widehat{B}}-1}
(\widehat{C}\widehat{B})^\pi(\widehat{C}\widehat{B})^{i}\widehat{D}^{(2i+2)d}\widehat{C}
+(\widehat{C}\widehat{B})^{d}\widehat{D}^{\pi}\widehat{C}
&
\displaystyle \sum_{i=0}^{i_{\widehat{C}\widehat{B}}-1}
(\widehat{C}\widehat{B})^\pi(\widehat{C}\widehat{B})^{i}\widehat{D}^{(2i+1)d}
\end{pmatrix}.
\end{align*}

Let
\[
\widehat{\phi}=\widehat{C}\widehat{B}
=CB+\varepsilon(CB_0+C_0B)=\phi+\varepsilon\phi_0.
\]
Then the dual-number formulas imply
\[
\widehat{\phi}^{D}=\phi^{D}+\varepsilon\phi_{R},
\qquad
\widehat{\phi}^{\pi}
=\phi^{\pi}-\varepsilon\bigl(\phi\phi_{R}+\phi_{0}\phi^{D}\bigr).
\]
Moreover, for $i\ge 1$, induction yields
\[
(\widehat{D}^{D})^{i}
=
D^{iD}
+\varepsilon\sum_{j=0}^{i-1}D^{jD}D_{R}D^{(i-j-1)D},
\qquad
\widehat{D}^{\,i}
=
D^{i}
+\varepsilon\sum_{j=1}^{i}D^{i-j}D_{0}D^{j-1}.
\]
Substituting these identities into the above expression of $\widehat{M}^D$ and simplifying term by term
yields the desired representation. This completes the proof.
\end{proof}

\begin{corollary}
    Let $\widehat{M}$ be given by \eqref{Dutch-windmill}. If $y_sx_t^T+\varepsilon(\widetilde{y_s}x_t^T+y_s\widetilde{x_t}^T)=0$, for $s,t\in[1,m]$. Then
\begin{align*}
    \widehat{M}^D=\begin{pmatrix}
        BD^{3D}C&BD^{2D}\\D^{2D}C&D^D
    \end{pmatrix}+\varepsilon\begin{pmatrix}
        BD^{3D}C_0+BD^{2D}D_RC+B_0D^{3D}C
        &BD^DD_R+BD_RD^D\\+BD_RD^{2D}C+BD^DD_RD^DC&+B_0D^{2D}\\[6pt]
        D^{2D}C_0+D^DD_RC+D_RD^DC&D_0
    \end{pmatrix}.
\end{align*}
\end{corollary}

McDonald et al. \cite{MNSEJAA2022} obtained a representation of the group inverse for the matrix in \eqref{bipartite} only, leaving the case of the matrix in \eqref{Dutch-windmill} unresolved. The results in this section extend their work by deriving explicit representations of the dual Drazin inverses for both matrices \eqref{bipartite} and \eqref{Dutch-windmill}. As a consequence, we present the following corollaries, which are directly from Theorem \ref{DWThm}.
\begin{corollary}
    Let $\widehat{M}$ be defined as in \eqref{Dutch-windmill}. Suppose that $\widehat{C}\widehat{B}$
    and $\widehat{D}$ are group invertible, and that, for all $s,t\in[1,m]$,
    $\widehat{D_{2n}}^{2s}(\widehat{D_{2n}}^s)^{\#}
    \widehat{y}_s\widehat{x}_t^T=0$, and
    $\widehat{D_{2n}}^s\widehat{y}_s
    \widehat{x}_t^T
    =\widehat{y}_s\widehat{x}_t^T
    \widehat{D_{2n}}^t(I-\widehat{D_{2n}}^t
    \widehat{D_{2n}}^{t\#})$. Then
    \begin{align*}
        \widehat{M}^{\#}=\begin{pmatrix}
            B(\phi^\pi D^{3\#}-\phi^\#D^\#-\phi^{2\#}DD^\pi )C&B\phi^\pi D^{2\#}+B\phi^\#D^\pi\\\phi^\pi D^{2\#}C+\phi^\# D^\pi C&\phi^\pi D^{\#}
        \end{pmatrix}+\varepsilon\begin{pmatrix}
            \eta_1&\eta_2\\\eta_3&\eta_4
        \end{pmatrix},
    \end{align*}where
    \begin{align*}
        \eta_1&=B\phi^\pi D^{3\#}C_0+B\phi^\pi D^{2\#}D_RC
        +B\phi^\pi D^{\#}D_RD^{\#}C
        +B\phi^\pi D_RD^{2\#}C+B_0\phi^\pi D^{3\#}C
        \\&-B\phi\phi_RD^{3\#}C-B\phi_0\phi^{\#}D^{3\#}C
        +B\phi^{2\#}D^2D_RC+B\phi^{2D}DD_0D^{\#}C
        -B\phi^{2\#}D_0D^\pi C
    \\&-B\phi^{2\#}DD^\pi C_0-B\phi^{\#}\phi_RDD^\pi C
    +B\phi_R\phi^{\#}DD^\pi C+B_0\phi^{2\#}DD^\pi C
    -B\phi^{\#}D^{\#}C_0
    \\&-B\phi^{\#}D_RC-B\phi_RD^{\#}C-B_0\phi^{\#}D^{\#}C-B_0\phi_RD^{\#}C,\\[4pt]
     \eta_2&=B\phi^\pi D^{\#}D_R+B\phi^\pi D_RD^{\#}+B_0\phi^\pi D^{2\#}-B\phi \phi_RD^{2\#}-B\phi_0\phi^{\#}D^{2\#}+B\phi^{\#}DD_R
     \\&+B\phi^{\#} D_0D^{\#}+B_0\phi^{\#} D^\pi +B\phi_RD^\pi,\\[4pt]
     \eta_3&=\phi^\pi D^{2\#}C_0+\phi^\pi D^{\#}D_RC+\phi^\pi D_RD^{\#}C-\phi\phi_RD^{2\#}C-\phi_0\phi^{\#}D^{2\#}C+\phi^{\#}D^\pi C_0
     \\&+\phi_RD^\pi C-\phi^{\#}DD_RC-\phi^{\#}D_0D^{\#}C,\\[4pt]
     \eta_4&=\phi^\pi D_R-\phi\phi_RD^{\#}-\phi_0\phi^DD^D.
     \end{align*}
\end{corollary}

\begin{corollary}
    Let $ \widehat{M}'$ be given by \eqref{bipartite}. Then
    \begin{align*}
        \widehat{M}'^D=\begin{pmatrix}
            O&E(FE)^D\\(FE)^DF&O
        \end{pmatrix}+\varepsilon\begin{pmatrix}
            O&E(FE)_R+E_0(FE)^D\\(FE)^DF_0+(FE)_RF&O
        \end{pmatrix}
    \end{align*}
\end{corollary}

The above corollary generalizes the result obtained in \cite{MNSEJAA2022}. For completeness, we recall it below.
\begin{corollary}
    Let $ \widehat{M}'$ be given by \eqref{bipartite} over $\mathbb{R}$. Then
    \begin{align*}
        \widehat{M}'^{\#}=\begin{pmatrix}
            O&(EF)^{\#}E\\F(EF)^{\#}&O
        \end{pmatrix}.
    \end{align*}
\end{corollary}
By Lemma \ref{cline}, we obtain $E(FE)^D=E(FE)^{2D}FE=(EF)^DE$ and $(FE)^DF=FE(FE)^{2D}F=F(EF)^D$, which reduce to the corresponding result for the group inverse when the index equals one.

\section*{Funding}
\begin{itemize}
\item Daochang Zhang is supported by the National Natural Science Foundation of China (NSFC) (No. 11901079), and China Postdoctoral Science Foundation (No. 2021M700751), and the Scientific and Technological Research Program Foundation of Jilin Province (No. JJKH20190690KJ; No. JJKH20220091KJ; No. JJKH20250851KJ).
\end{itemize}

\section*{Conflict of Interest}
The authors declare that they have no potential conflict of interest.

\section*{Data Availability}
Data sharing not applicable to this article as no datasets were generated or analyzed during the current study.


\bibliographystyle{abbrv}

\end{document}